\documentclass[a4paper,twoside,11pt]{article}

\usepackage{a4wide, fancyhdr, amsmath, amssymb, mathtools, yfonts}
\usepackage{mathrsfs}
\usepackage{graphicx}
\usepackage{tikz}
\usepackage[all]{xy}
\usepackage[utf8]{inputenc}
\usepackage{amsthm}
\usepackage[english]{babel}
\usepackage{chngcntr}
\usepackage{ifthen}
\usepackage{calc}
\usepackage{hyperref}
\usepackage{authblk}
\numberwithin{equation}{section}
\newcommand{\Z}{\mathbb{Z}}
\newcommand{\Q}{\mathbb{Q}}

\newcommand\FF{\mathbb{F}}

\newtheorem{lemma}{Lemma}[section]
\newtheorem{theorem}[lemma]{Theorem}
\newtheorem{proposition}[lemma]{Proposition}

\newtheorem{definition}[lemma]{Definition}

\newtheorem{remark}{Remark}

\title{\vspace{ - \baselineskip}\sffamily\bfseries Effective convergence of coranks of random R\'edei matrices}
\author[1]{Peter Koymans\thanks{Vivatsgasse 7, 53111  Bonn, Germany, koymans@mpim-bonn.mpg.de}}
\author[1]{Carlo Pagano\thanks{Vivatsgasse 7, 53111  Bonn, Germany, carlein90@gmail.com}}
\affil[1]{Max Planck Institute for Mathematics, Bonn}
\date{\today}

\begin{document}
\maketitle

\begin{abstract}
We give effective estimates for the $l^1$-distance between the corank distribution of $r \times r$ R\'edei matrices and the measure predicted by the Cohen--Lenstra heuristics. To this end we pinpoint a class of stochastic processes, which we call $c$-transitioning. These stochastic processes are well approximated by Markov processes, and we give an effective ergodic theorem for such processes. With this tool we make effective a theorem of Gerth \cite{Gerth} that initiated the study of the Cohen--Lenstra heuristics for $p = 2$. 

Gerth's work triggered a series of developments that has recently culminated in the breakthrough of Smith \cite{Smith}. The present work will be used in upcoming work of the authors on further applications of Smith's ideas to the arithmetic of quadratic fields. To this end we extend our main result to several other families of matrix spaces that occur in the study of integral points on the equation $x^2 - dy^2 = l$ as $d$ varies.
\end{abstract}

\section{Introduction}
In $1983$ Cohen and Lenstra \cite{cohen--lenstra} put forth a systematic set of conjectures on the distribution of the $p$-Sylow of class groups of quadratic fields, for an odd prime $p$. These conjectures postulate that these $p$-Sylows should behave as randomly as possible in the following natural sense. Each finite abelian $p$-group $A$ is conjectured to appear as the $p$-Sylow of the class group of a quadratic field with probability proportional to $\frac{1}{|\text{Aut}(A)|}$ and $\frac{1}{|\text{Aut}(A)|\cdot |A|}$ for, respectively, imaginary and real quadratic fields. 

A substantial amount of effort has subsequently been invested in detecting a possible source of randomness in the behavior of class groups of quadratic fields. In $1987$ Friedman and Washington \cite{Friedman-Washington} considered the case of quadratic function fields and observed that the Cohen--Lenstra's prediction could be reformulated in terms of random matrices. Indeed, they suggested as potential source of randomness the behavior of the Frobenius operator acting on the Tate module of the corresponding hyperelliptic curve. 

This lead to a reinterpretation of the Cohen--Lenstra heuristics in terms of corank statistics of large random matrices, a point of view that has been further explored in the work of Wood \cite{Wood1} and proved fruitful in her recent extension of the conjectures of Cohen and Lenstra to a non-abelian setting \cite{Wood2}. Incidentally, the matrices occurring over function fields are constrained to respect a symplectic pairing (the Weil pairing), but it can be shown that this does not affect the limiting distribution, a feature that plays a major role also in the present work. In the context of quadratic function fields Ellenberg, Venkatesh and Westerland \cite{EVW} were able to make substantial progress relating these conjectures to the homological stability of Hurwitz spaces. They then used topological methods to make progress on the latter.  

For quadratic number fields the situation is currently much more mysterious if $p$ is an odd prime, apart from the average $3$-torsion of quadratic fields \cite{BST, DH}. As we shall now explain, the story for $p = 2$ is entirely different.  

In $1801$ Gauss \cite{Gauss} gave a description of the $2$-torsion of the narrow class group $\text{Cl}(\Q(\sqrt{d}))$ of a quadratic field $\mathbb{Q}(\sqrt{d})$. In particular Gauss showed that 
\[
\dim_{\FF_2} \text{Cl}(\Q(\sqrt{d}))[2] = \omega(\Delta_{\mathbb{Q}(\sqrt{d})}) - 1,
\]
where $\omega(\cdot)$ denotes the number of distinct prime divisors and $\Delta_{\mathbb{Q}(\sqrt{d})}$ is the discriminant of $\mathbb{Q}(\sqrt{d})$. Recently, the authors \cite{KP2} investigated the $2$-torsion of the narrow class group of multiquadratic fields $\mathbb{Q}(\sqrt{d_1}, \ldots, \sqrt{d_n})$. 

This readily shows that the $2$-Sylow of the class group of a quadratic field is not a random finite, abelian $2$-group in the sense of Cohen--Lenstra. A natural guess, which can be found implicitly in \cite{Gerth} and explicitly in \cite{Gerth3}, is that instead the group $2\text{Cl}(\mathbb{Q}(\sqrt{d}))[2^{\infty}]$ is a random finite, abelian $2$-group.

In $1984$ Gerth \cite{Gerth} gave the first evidence for the correctness of this guess. During the proof he made use of an explicit description of $2\text{Cl}(\mathbb{Q}(\sqrt{d}))[4]$ due to R\'edei \cite{Redei}. R\'edei was able to relate the dimension of the space $2\text{Cl}(\mathbb{Q}(\sqrt{d}))[4]$ to the corank of a certain $r \times r$ matrix constructed out of the mutual Legendre symbols of the prime divisors of $d$, where
$$
r := \omega(\Delta_{\mathbb{Q}(\sqrt{d})/\mathbb{Q}}).
$$
This matrix is now commonly referred to as the \emph{R\'edei matrix} of the field $\mathbb{Q}(\sqrt{d})$. Due to quadratic reciprocity the R\'edei matrix is far from being a random $r \times r$ matrix. Gerth showed that when one \emph{fixes} $r$, then the R\'edei matrices of $\mathbb{Q}(\sqrt{d})$, as $d$ varies, equidistribute in the space of $r \times r$ matrices satisfying the constraints of quadratic reciprocity. He then showed that the corresponding densities converge to the limiting distribution predicted by the above guess as $r$ goes to infinity, despite the constrained shape of the matrices. Gerth also extended his work to cyclic degree $p$ extensions \cite{Gerth2} and formulated a conjecture for the distribution of their $p$-Sylows, which is, in case $p = 2$, precisely the modified version of the Cohen--Lenstra conjectures mentioned above. This then became known as Gerth's conjecture or the Cohen--Lenstra--Gerth heuristic.

Despite Gerth's progress, the distribution of the $4$-torsion as $d$ varies among all squarefree integers was still an open problem. In $2006$ Fouvry and Kl\"uners \cite{FK1, fouvry--kluners}, using a different approach, were able to solve this problem and showed that the $4$-rank of class groups of quadratic fields have the limiting distribution predicted by the Cohen--Lenstra--Gerth heuristic. Instead of directly trying to handle the randomness of the R\'edei matrices, they expressed the $4$-rank of $\mathbb{Q}(\sqrt{d})$ as a sum of Legendre symbols. They proved oscillation of this sum by using ideas introduced in the seminal work of Heath--Brown on $2$-Selmer groups \cite{HB}.

On the one hand this method has proved to be very robust for the $4$-torsion and analogous statistics: a similar line of attack has been subsequently used for cyclic degree $p$ fields \cite{Klys}, for ray class groups of imaginary quadratic fields \cite{Pagano--Sofos} and very recently by Fouvry and the authors \cite{FKP} for the $4$-rank of $\text{Cl}(\mathbb{Q}(i, \sqrt{d}))$. On the other hand the situation for the $8$-torsion and higher powers remained mysterious until very recently. 

In $2017$ Smith \cite{Smith}, improving on earlier work of himself on the $8$-torsion \cite{smith1}, was able to prove Gerth's conjecture for imaginary quadratic fields in its entirety. In $2018$ this was extended by the authors \cite{KP1} to all cyclic degree $p$ extensions, conditional on GRH. 

In Smith's work the description of the $4$-torsion in terms of R\'edei matrices becomes again central. In \cite{Smith} he manages to prove that the R\'edei matrices attached to squarefree integers $d$ is, for the purposes of the corank statistics, equidistributed in the space of all possible R\'edei matrices, when one lets $d$ run through all squarefree integers. In this context he claims that the rate of convergence in the main result of Gerth \cite[Theorem 4.3]{Gerth} can be made effective. The main result of the present work fills this gap in the literature by showing that this is indeed the case. 

We prove effective convergence of the corank distribution of a large random R\'edei matrix to the probability distribution predicted by Cohen--Lenstra--Gerth. For an integer $0 \leq \kappa \leq r$, let $X_r(\text{R\'edei}, \kappa)$ be the random variable that computes the probability that a uniformly chosen $r \times r$ matrix $A$ with coefficients in $\FF_2$ satisfying 
\[
A(i, h) = A(h, i) + 1 \text{ for all } 1 \leq i < h \leq \kappa, \quad A(i, h) = A(h, i) \text{ for all } \kappa < h \leq r \text{ and } 1 \leq i \leq r
\]
has corank equal to $j$. For an integer $r \geq 0$ denote by $\mu_r(\text{R\'edei})$ the probability distribution on $\mathbb{Z}_{\geq 0}$ given by the corank of a $r \times r$ random R\'edei matrix, i.e.
\[
\mu_r(\text{\'Redei})(j) = \frac{1}{2^r} \sum_{\kappa = 0}^r \binom{r}{\kappa} \mathbb{P}(X_r(\text{R\'edei}, \kappa) = j).
\]
Denote by $\pi_{\text{C.L.}} : \Z_{\geq 0} \rightarrow \mathbb{R}_{\geq 0}$ the distribution of the rank of the $2$-torsion of random abelian $2$-groups, in the sense of Cohen--Lenstra--Gerth. Writing
\[
\eta_k(t) := \prod_{j = 1}^k (1 - t^{-j})
\]
for $k \in \Z_{\geq 0} \cup \{\infty\}$, we have the explicit formula
\begin{align}
\label{eCL}
\pi_{\text{C.L.}}(j) = 2^{-j^2} \eta_\infty(2) \eta_j(2)^{-2},
\end{align}
which equals the probability that a uniformly chosen random $r \times r$ matrix with coefficients in $\FF_2$ has corank $j$, as $r$ goes to infinity.

\begin{theorem} 
\label{main}
There exists $C \in \mathbb{R}_{>0}$ and $\rho \in (0,1)$ such that
$$
||\mu_r(\emph{R\'edei}) - \pi_{\emph{C.L.}}||_1 \leq C \cdot \rho^r
$$
holds for every $r \in \mathbb{Z}_{\geq 0}$. 
\end{theorem} 

As explained in Remark \ref{rEGerth}, Theorem \ref{main} is an effective version of Gerth's main result \cite[Theorem 4.3]{Gerth}. The method of proof of Theorem \ref{main}, which we summarize below, can be adapted to other matrix spaces. For instance, in Theorem \ref{limiting distribution for Pellian families} we extend this result to the case of spaces of R\'edei matrices occurring in the study of the solubility of the equation
$$
x^2 - dy^2 = l
$$
with $x,y \in \mathbb{Z}$. Here $l$ is fixed and $d$ varies over squarefree integers divisible by $l$.

We remark that a similar analysis is not required in the case of cyclic degree $p$ fields, where $p$ is an odd prime. The difference is explained by quadratic reciprocity. This highly constraints the space of R\'edei matrices of quadratic fields. However, the key point, already present in Gerth's work \cite{Gerth}, is that for \emph{most} $r \times r$ R\'edei matrices $A$ adding a random row and column to $A$, in a way that the resulting $(r+1) \times (r+1)$ matrix is still R\'edei, the corank transitions with the same probabilistic rules as that of a random matrix. At this point Gerth proceeds with a detailed analysis of the transition rules for the exceptions and construct several explicit Markov processes that allow him to obtain the desired limiting distribution by an intricate approximation argument. 

We completely bypass this intricate step and take the following route instead. We pinpoint the general class of $c$-transitioning processes, which are well approximated by a Markov process in a precise quantitative sense. Then we give an effective ergodic theorem for such processes in Theorem \ref{ergodic with imperfect info} with uniform error term. After that, we rapidly deduce Theorem \ref{main} and several analogues.
\subsection*{Acknowledgements}
We thank Adam Morgan for fruitful discussions. Both authors are grateful to the Max Planck Institute for Mathematics in Bonn for its hospitality and financial support.

\section{An ergodic theorem for Markov chains}
We shall need to work with Markov chains of considerable generality to prove our main theorems. Fortunately, the relevant Markov chains are still simple enough that we shall not need too much machinery from measure theory. Let $\Omega$ be a countable set, which we view as a measurable space by equipping it with the $\sigma$-algebra consisting of all subsets of $\Omega$. In this way we can think of measures simply as functions $\Omega \rightarrow \mathbb{R}_{\geq 0}$ and we shall often do so implicitly.

For every $x \in \Omega$, there is a natural random variable $X(x)$ on $\Omega$, which assigns to $x$ probability $1$ and $0$ to all other points. Let $P: \Omega \times \Omega \rightarrow \mathbb{R}_{\geq 0}$ be a function such that $\{x : P(x, y) > 0\}$ is finite for all $y \in \Omega$, $\{y : P(x, y) > 0\}$ is finite for all $x \in \Omega$ and
\[
\sum_{y \in \Omega} P(x, y) = 1
\]
for all $x \in \Omega$. We can think of $P$ as an infinite matrix with only finitely many non-zero entries in each row and column such that the sum of the entries in every row is $1$, and we call such $P$ transition matrices. Furthermore, if $\mu : \Omega \rightarrow \mathbb{R}_{\geq 0}$ is a probability measure, we can left multiply
\[
(\mu P)(y) = \sum_{x \in \Omega} P(x, y) \mu(x)
\]
to obtain another probability measure. The Markov chains that we shall encounter will start with a random variable $X(x)$ for some $x \in \Omega$, and the next random variables are obtained by repeated application of the same $P$ satisfying the assumptions above. For $A \subseteq \Omega$, we write $P^n(x, A)$ for the probability that the Markov chain is in $A$ after $n$ steps, assuming that the Markov chain starts in $x$, i.e. with the random variable $X(x)$.

Let $\psi: \Omega \rightarrow \mathbb{R}_{\geq 0}$ be a measure such that $\psi(\Omega) > 0$. We say that $P$ is $\psi$-irreducible if for every $x \in \Omega$ and every $A \subseteq \Omega$ with $\psi(A) > 0$, there is some positive integer $n$ such that $P^n(x, A) > 0$. We say that $P$ is aperiodic if
\[
\gcd(\{n \geq 1 : P^n(x, x) > 0\}) = 1 \text{ for all } x \in \Omega
\]
and extremely aperiodic if
\[
P(x, x) > 0 \text{ for all } x \in \Omega.
\]
Certainly, if $P$ is extremely aperiodic, then it is also aperiodic. We say that a function $V : \Omega \rightarrow \mathbb{R}_{\geq 1}$ is a drift function if there is some $\lambda < 1$ such that
\[
PV(x) \leq \lambda V(x)
\]
for all but finitely many $x \in \Omega$ and furthermore
\[
|\{x \in \Omega : V(x) \leq r\}| < \infty
\]
for every real number $r$. Here $PV$ denotes the function obtained by right multiplying $P$ with $V$.

\begin{theorem} 
\label{ergodic theorem}
Let $\Omega$ be a countable set, and let $P: \Omega \times \Omega \rightarrow \mathbb{R}_{\geq 0}$ be a transition matrix. Assume that $P$ is $\psi$-irreducible and extremely aperiodic. Then if $V: \Omega \rightarrow \mathbb{R}_{\geq 1}$ is a drift function, there is a unique probability measure $\pi: \Omega \rightarrow \mathbb{R}_{\geq 0}$ such that $\pi P = P$. Furthermore, there are constants $R > 0$ and $\rho < 1$ such that
\[
\sum_{y \in \Omega} \left|P^n(x, y) - \pi(y)\right| \leq RV(x)\rho^n
\]
for every $x \in \Omega$.
\end{theorem}

\begin{proof}
Let $C$ be the finite set of exceptions to the inequality
\[
PV(x) \leq \lambda V(x).
\]
Since our Markov chain is extremely aperiodic, it follows that any finite subset of $\Omega$ is petite, see Section 5.5 of \cite{MT} for the definition of petite. In particular, $C$ is petite. Then condition (iii) of Theorem 15.0.1 in \cite{MT} is satisfied (in their notation $\Delta V(x) := PV(x) - V(x)$). It follows from \cite[Theorem 15.0.1]{MT} that $\pi$ exists, and that there are constants $R > 0$ and $r > 1$ such that
\begin{equation}
\label{eErgodic}
\sum_{n = 1}^\infty r^n ||P^n(x, \cdot) - \pi||_V \leq RV(x).
\end{equation}
Here $||\cdot||_V$ is by definition
\[
||g||_V = \sup_{f : |f| \leq V} \left|\sum_{y \in \Omega} f(y) \cdot g(y)\right|.
\]
Note that equation (\ref{eErgodic}) implies that there are $R' > 0$ and $\rho < 1$ such that
\[
||P^n(x, \cdot) - \pi||_V \leq R'V(x)\rho^n.
\]
This proves the theorem by choosing $f$ to be the function with $|f| = 1$ and $f(y) > 0$ if and only if $P^n(x, y) - \pi(y) > 0$.
\end{proof}

\section{The equilibrium for almost transitioning processes} 
\label{ergodic for c-transitioning}
Let $Q$ be a transition matrix on $\mathbb{Z}_{\geq 0}$. Let $d$ be a positive real number. We say that $Q$ is $d$-driftable in case $x \mapsto d^x$, viewed as map from $\mathbb{Z}_{\geq 0}$ to $\mathbb{R}_{\geq 0}$, is a drift function for $Q$. We further assume that $Q$ is $\psi$-irreducible for some non-trivial measure $\psi$ on $\mathbb{Z}_{\geq 0}$ and that $Q$ is extremely aperiodic. Observe that multiplication by $Q$ on the left on $l^{1}(\mathbb{Z}_{\geq 0})$ gives a bounded linear operator with 
$$
||Q||_1=1.
$$
Indeed given $v \in l^1(\mathbb{Z}_{\geq 0})$ with $||v||_1=1$, we have that, writing $w$ for the vector obtained from $v$ by taking absolute values componentwise,
$$
||vQ||_1 \leq ||wQ||_1,
$$
since the entries of $Q$ are non-negative. On the other hand $||wQ||_1=1$ because $Q$ is a transition matrix. 

Let $(A, \mathbb{P})$ be a probability space (the $\sigma$-algebra will not play a role and hence we do not introduce notation for it). Let $c$ be a real number in $(0,1)$. Suppose to have for each integer $i \geq 0$ random variables
$$
X_i: A \to \{0, \dots, i\}.
$$
We say that the sequence $\{X_i\}_{i \in \mathbb{Z}_{\geq 0}}$ is $c$-transitioning with $Q$ in case there exists a sequence of random variables
$$
Z_i: A \to \{0,1\}
$$
such that $\mathbb{P}(Z_i = 1) \leq c^i$ and
$$
\mathbb{P}(X_{i + 1} = j | X_i = s, Z_i = 0) = Q(s, j)
$$
for each $s \in \{0, \dots, i\}$ and $j \in \{0, \dots, i + 1\}$.

For a random variable $X : A \to \mathbb{Z}_{\geq 0}$, write $\mu_X$ for the vector in $l^1(\mathbb{Z}_{\geq 0})$ given by the distribution of $X$, that is
$$
\mu_X(j) := \mathbb{P}(X = j).
$$

\begin{theorem}
\label{ergodic with imperfect info}
Let $d \in \mathbb{R}_{ \geq 0}$ and let $c \in (0,1)$. Let $Q$ be a $d$-driftable transition matrix on $\mathbb{Z}_{\geq 0}$. Let $\{X_i\}_{i \in \mathbb{Z}_{\geq 0}}$ be a sequence of random variables, with $X_i$ taking values in $\{0, \dots, i\}$. Suppose that $\{X_i\}_{i \in \mathbb{Z}_{\geq 0}}$ is $c$-transitioning with $Q$. Then there exists a unique probability measure $\pi$ on $\mathbb{Z}_{\geq 0}$ with $\pi Q=\pi$ and constants $C_{Q,c,d} \in \mathbb{R}_{>0}, \rho_{Q,c,d} \in (0,1)$ such that 
$$
||\mu_{X_r} - \pi||_1 \leq C_{Q,c,d} \cdot \rho_{Q,c,d}^r
$$
for each $r \in \mathbb{Z}_{\geq 0}$.
\end{theorem}

\begin{proof}
Let us define two sequences of vectors $\{v_i,w_i\}_{i \in \mathbb{Z}_{\geq 0}}$ with $v_i, w_i \in l^1(\mathbb{Z}_{\geq 0})$ in the following manner
$$
v_i(j) := \mathbb{P}(X_i = j, Z_i = 1)
$$
and
$$
w_i(j) := \mathbb{P}(X_{i + 1} = j, Z_i = 1)
$$
for each $j \in \mathbb{Z}_{\geq 0}$. A simple calculation, using that our process transitions correctly when $Z_i = 0$, gives that
$$
(\mu_{X_i} - v_i)Q + w_i = \mu_{X_{i + 1}}
$$
for each $i \in \mathbb{Z}_{\geq 0}$.  

Applying iteratively this identity we find that for each $i \in \mathbb{Z}_{\geq 0}$ and each $h \in \mathbb{Z}_{\geq 0}$
$$
\mu_{X_{i + h}} = \mu_{X_i} Q^h - \sum_{s = 0}^{h - 1} v_{i + s} Q^{h - s} + \sum_{s = 0}^{h - 1} w _{i + s} Q^{h - s - 1}.
$$
Let us now pick $\epsilon$ in $(0,1)$ and write $g(r) := \lfloor \epsilon \cdot r \rfloor$ and $h(r) := r - g(r)$. By the triangle inequality we have
$$
||\mu_{X_r} - \pi||_1 \leq ||\mu_{X_{g(r)}} Q^{h(r)} - \pi||_1 + \sum_{s = g(r)}^{r - 1} ||v_s Q^{r - s}||_1 + \sum_{s = g(r)}^{r - 1} ||w_s Q^{r - s -1}||_1.
$$
Thanks to Theorem \ref{ergodic theorem}, there are $R \in \mathbb{R}_{\geq 0}$ and $\rho \in (0,1)$ depending entirely on $Q$ and $d$ such that the first summand is bounded by 
$$
R \cdot d^{g(r)} \cdot \rho^{h(r)}.
$$
Now choose $\epsilon \in (0,1)$ so that $d^{\epsilon} \rho^{1 - \epsilon} < 1$. This choice makes the expression smaller that $C \cdot \rho'^r$ for some $\rho'$ depending only on $Q$ and $d$. Now we focus on
$$
\sum_{s = g(r)}^{r - 1} (||v_sQ^{r-s}||_1+||w_sQ^{r-s-1}||_1).
$$
Using that $||Q||_1=1$ and that
$$
||v_s||_1 = ||w_s||_1 \leq c^s,
$$
we find the upper bound
$$
2 \cdot \sum_{s = g(r)}^{r - 1} c^s = O_c(c^{g(r)}) = O_c(c^{\epsilon r}). 
$$
This gives the desired conclusion. 
\end{proof}

\begin{remark} 
\label{transitions at finite level}
From the proof, it is clear that we reach the same conclusion of Theorem \ref{ergodic with imperfect info} on $||\mu_{X_r} - \pi||_1$ as long as we have the definition of $c$-transitioning for all indices up to $r$, the ones after $r$ being clearly irrelevant for the estimate at stage $r$. This point will be important in the proof of Theorem \ref{limiting distribution for redei matrices}.
\end{remark}

\section{Corank distributions of matrix spaces}
We now study the rank distribution in a number of matrix spaces. Such spaces, often occurring in arithmetic applications, arise by randomly adding to a given matrix a row and a column subject to certain \emph{rules}. We formalize this notion as follows. Write $\FF_2^{[n]}$ for the free $\FF_2$ vector space over $[n]$. For a subset $S \subseteq [n] := \{1, \dots, n\}$, we write $\pi_S$ for the natural projection map.

\begin{definition}
A rule is a product space
$$ 
\mathcal{S} := \prod_{i \in \mathbb{Z}_{\geq 0}} S_i,
$$
where $S_i$ is a non-empty subset of $\mathbb{F}_2^{[i]} \times \mathbb{F}_2^{[i]}  \times \mathbb{F}_2$ for each $i$ in $\mathbb{Z}_{\geq 0}$. 
\end{definition}

We can view each $S_i$ as probability space with uniform probability distribution and as discrete topological space. Thanks to the Kolmogorov extension theorem \cite[Theorem 2.4.3]{Tao}, this naturally endows a rule $\mathcal{S}$ with the structure of a probability space (where the sigma algebra is the one generated by the open sets of $\mathcal{S}$, viewed as profinite space).

We shall often use the following construction. To each point $B \in \mathcal{S}$ one can naturally attach an infinite matrix
$$
\omega(B) \in \text{Mat}_{\mathbb{F}_2}(\mathbb{Z}_{\geq 1} \times \mathbb{Z}_{\geq 1}),
$$
together with its sequence of top left minors $\omega_i(B) \in \text{Mat}_{\mathbb{F}_2}([i] \times [i])$. This gives a sequence of random variables
$$
X_i: \mathcal{S} \to \{0, \dots, i\}
$$
given by $B \mapsto \text{co-rk}(\omega_i(B))$. 

Below we consider several different rules $\mathcal{S}$. These rules will give rather different matrix spaces. However, each rule has in common that we are able to find a \emph{generic} class of matrices such that the corank transitions precisely as in the simplest case, namely the class of random matrices. This is given by the transition matrix
\[
Q_{\text{C.L.}}(i, j) =
\left\{
	\begin{array}{ll}
		1 + 2^{- 2i} - 2^{1 - i} & \mbox{if } j = i - 1 \\
		2^{1 - i} - 3 \cdot 2^{-1 - 2i} & \mbox{if } j = i \\
		2^{-1 - 2i} & \mbox{if } j = i + 1
	\end{array}
\right.
\]
and zero otherwise. The matrix $Q_{\text{C.L.}}$ is $\psi$-irreducible for the function $\psi(x) = 1$, and extremely aperiodic. Furthermore, $Q_{\text{C.L.}}$ is $2$-driftable with the exceptional states being $\{0, 1, 2\}$.

This matrix will play the role of $Q$ in Theorem \ref{ergodic with imperfect info}. The rule $\mathcal{S}$ will play the role of $A$ from Section \ref{ergodic for c-transitioning}, and the variables $Z_i$ from that same section will precisely be the detector of genericity, which we define in each space. A direct computation using equation (\ref{eCL}) shows that
$$
\pi_{\text{C.L.}}Q_{\text{C.L.}} = \pi_{\text{C.L.}}.
$$ 
In this way the effective convergence will fall as a formal consequence of Theorem \ref{ergodic with imperfect info}.

\subsection{Rank transition in row-column extension of a matrix}
Let $n$ be a positive integer and let $A \in \text{Mat}_{\mathbb{F}_2}([n] \times [n])$. We denote by $<,>$ the standard inner product of two vectors in $\mathbb{F}_2^{[n]}$. A vector $w$ is in $\text{Im}(A^T)$ if and only if it is in $\text{ker}(A)^{\perp}$. Therefore, if $v$ is in $\text{Im}(A)$, the set $<A^{-1}(v),w>$ consists of a unique number, which by abuse of notation we also denote as $<A^{-1}(v),w>$. 
 
Given $v,w \in \mathbb{F}_2^{[n]}$ and $c \in \mathbb{F}_2$, we denote by $A(v,w,c)$ the matrix in $\text{Mat}_{\mathbb{F}_2}([n+1] \times [n+1])$ given by
\[
(i, j) \mapsto
\left\{
	\begin{array}{ll}
		A(i, j) & \mbox{if } (i, j) \in [n] \times [n] \\
		v(j) & \mbox{if } i = n + 1, j \in [n] \\
		w(i) & \mbox{if } j = n + 1, i \in [n] \\
		c & \mbox{if } i = j = n + 1.
	\end{array}
\right.
\]
The following fact describes the corank transition $\text{co-rk}(A) \mapsto \text{co-rk}(A(v,w,c))$.

\begin{proposition} 
\label{rank transition}
Let $A$ be in $\emph{Mat}_{\mathbb{F}_2}([n] \times [n])$, $v,w \in \mathbb{F}_2^{[n]}$ and $c \in \mathbb{F}_2$. Then one has the following: \\
$(a)$ $\emph{co-rk}(A,v,w,c)=\emph{co-rk}(A)+1$ if and only if $v \in \emph{Im}(A), w \in \emph{Im}(A^{T})$ and $c=<A^{-1}v,w>$. \\
$(b)$ $\emph{co-rk}(A,v,w,c)=\emph{co-rk}(A)-1$ if and only if $v \not \in \emph{Im}(A), w \not \in \emph{Im}(A^{T})$.\\
$(c)$ $\emph{co-rk}(A,v,w,c)=\emph{co-rk}(A)$ in all remaining cases.  
\end{proposition}

\begin{proof}
This follows from basic linear algebra, see Gerth \cite{Gerth}.
\end{proof}

We write $H_n$ for the vector in $\mathbb{F}_2^{[n]}$ with all entries equal to $1$. 

\begin{proposition} 
\label{transition as CL}
Fix $A$ in $\emph{Mat}_{\mathbb{F}_2}([n] \times [n])$. Denote by $j:=\emph{co-rk}(A)$. Then we have the following. \\
$(1)$ Picking $(v,w,c)$ uniformly at random in $\mathbb{F}_2^{[n]} \times \mathbb{F}_2^{[n]} \times \mathbb{F}_2$ the random variable $\emph{co-rk}(A(v,w,c))$ takes the values $\{j - 1, j, j + 1\}$ with probability given respectively by $Q_{\emph{C.L.}}(j, j - 1)$, $Q_{\emph{C.L.}}(j, j)$ and $Q_{\emph{C.L.}}(j, j + 1)$. \\
$(2)$ Picking $(v,c)$ uniformly at random in $\mathbb{F}_2^{[n]} \times \mathbb{F}_2$, the random variable $\emph{co-rk}(A(v,v,c))$ takes the values $\{j - 1, j, j + 1\}$ with probability given respectively by $Q_{\emph{C.L.}}(j, j - 1)$, $Q_{\emph{C.L.}}(j, j)$ and $Q_{\emph{C.L.}}(j, j + 1)$ if and only if 
$$
\emph{ker}(A) \cap \ker(A^T)=\{0\}.
$$ 
$(3)$ Picking $(v,c)$ uniformly at random in $\mathbb{F}_2^{[n]} \times \mathbb{F}_2$, the random variable $\emph{co-rk}(A(v,v+H_n,c))$ takes the values $\{j - 1, j, j + 1\}$ with probability given respectively by $Q_{\emph{C.L.}}(j, j - 1)$, $Q_{\emph{C.L.}}(j, j)$ and $Q_{\emph{C.L.}}(j, j + 1)$ if 
$$
\emph{ker}(A) \cap \ker(A^T)=\{0\}.
$$
\end{proposition}

\begin{proof}
This is a simple consequence of Proposition \ref{rank transition} and the next remark.
\end{proof}

\begin{remark} 
\label{images span everything}
We have
\[
\textup{ker}(A) \cap \textup{ker}(A^T) = \{0\} \Longleftrightarrow \textup{Im}(A) + \textup{Im}(A^T) = \mathbb{F}_2^{[n]}.
\]
Indeed, this follows immediately after applying $\perp$ and using that $\textup{Im}(A^T) = \textup{ker}(A)^{\perp}$.
\end{remark}

\subsection{Random matrices}
Let us consider the rule $\mathcal{S}_{\text{mat}}$ defined by
$$
S_i(\text{mat}):=\mathbb{F}_2^{[i]} \times \mathbb{F}_2^{[i]} \times \mathbb{F}_2
$$
and let $X_i(\text{mat}) : \mathcal{S}_{\text{mat}} \rightarrow \{0, \dots, i\}$ be the corresponding sequence of random variables.

\begin{theorem} 
\label{limiting distribution for random matrices}
There exists $C \in \mathbb{R}_{>0}$ and $\rho \in (0,1)$ such that
$$
||\mu_{X_r(\emph{mat})} - \pi_{\emph{C.L.}}||_1 \leq C \cdot \rho^r,
$$
for each $r \in \mathbb{Z}_{\geq 0}$. 
\end{theorem}

\begin{proof}
This follows immediately upon combining part $(1)$ of Proposition \ref{transition as CL} and Theorem \ref{ergodic theorem}.
\end{proof}

In this case there are also explicit formulas available for $\mu_{X_r(\text{mat})}$, which also allow one to deduce Theorem \ref{limiting distribution for random matrices}. We have included this case as the simplest illustration of the methods used here.

\subsection{Alternating matrices}
Let us consider the rule $\mathcal{S}_{\text{alt}}$ defined by
$$
S_i(\text{alt}) := \{(v,w,c) \in \mathbb{F}_2^{[i]} \times \mathbb{F}_2^{[i]} \times \mathbb{F}_2: v + w=H_i\}
$$
and let $X_i(\text{alt}) : \mathcal{S}_{\text{alt}} \rightarrow \{0, \dots, i\}$ be the corresponding sequence of random variables.

\begin{theorem} 
\label{limiting distribution for alternating matrices}
There exists $C \in \mathbb{R}_{>0}$ and $\rho \in (0,1)$ such that
$$
||\mu_{X_r(\emph{alt})} - \pi_{\emph{C.L.}}||_1 \leq C \cdot \rho^r,
$$
for each $r \in \mathbb{Z}_{\geq 0}$.
\end{theorem}

\begin{proof}
For each integer $i \geq 0$ let us define
$$
Z_i(\text{alt}): \mathcal{S}_{\text{alt}} \to \{0,1\}
$$
to be the assignment $B \mapsto 0$ in case $\omega_i(B) \cdot H_i \neq 0$ and $B \mapsto 1$ otherwise. Clearly
$$
\mathbb{P}(Z_i(\text{alt}) = 1) = \frac{1}{2^i}.
$$
Furthermore, by Proposition \ref{transition as CL} part $(3)$ and Remark \ref{images span everything} we deduce that for each $j \in \{0, \dots, i + 1\}$ and $k \in \{0, \dots, i\}$ we have that
$$
\mathbb{P}(X_{i + 1}(\text{alt}) = j | X_i(\text{alt}) = k, Z_i = 0) = Q_{\text{C.L.}}(k, j). 
$$
Hence the desired conclusion falls as an immediate consequence of Theorem \ref{ergodic with imperfect info}.
\end{proof}

\subsection{R\'edei matrices} 
\label{Redei matrices}
Let us consider the rule $\mathcal{S}_{\text{R\'edei}}(\kappa)$ defined by
$$
S_i(\text{R\'edei})(\kappa):=\{(v,w,c) \in \mathbb{F}_2^{[i]} \times \mathbb{F}_2^{[i]} \times \mathbb{F}_2: v+w=H_i \},
$$
if $i < \kappa$ and 
$$ 
S_i(\text{R\'edei})(\kappa):=\{(v,w,c) \in \mathbb{F}_2^{[i]} \times \mathbb{F}_2^{[i]} \times \mathbb{F}_2: v=w \},
$$
if $i \geq \kappa$. This gives a sequence of random variables
$$
X_i(\text{R\'edei}, \kappa): \mathcal{S}_{\text{R\'edei}}(\kappa) \to \{0, \dots, i\},
$$
given by $B \mapsto \text{co-rk}(\omega_i(B))$. 

Let now $r$ be in $\mathbb{Z}_{\geq 0}$. In arithmetic applications one considers the space of R\'edei matrices where firstly $\kappa$ is chosen with probability $\text{Binom}(r, \kappa):=2^{-r}$ ${r} \choose {\kappa}$ and then the matrix is chosen as $\omega_r(B)$ with $B \in \mathcal{S}_{\text{R\'edei}}(\kappa)$ chosen randomly. The resulting probability distribution is
$$
\mu_r(\text{R\'edei}):=\sum_{\kappa=0}^{r}\text{Binom}(r, \kappa) \cdot \mu_{X_{r}(\text{R\'edei}, \kappa)}.
$$
We now prove the following.

\begin{theorem} 
\label{limiting distribution for redei matrices}
There exists $C \in \mathbb{R}_{>0}$ and $\rho \in (0,1)$ such that 
$$
||\mu_r(\emph{R\'edei}) - \pi_{\emph{C.L.}}||_1 \leq C \cdot \rho^r.
$$
\end{theorem}

\begin{proof}
Fix $\epsilon \in (0, \frac{1}{2})$. Then we have
\begin{multline*}
||\mu_r(\text{R\'edei}) - \pi_{\text{C.L.}}||_1 \leq \sum_{\epsilon \cdot r \leq \kappa \leq (1 - \epsilon) \cdot r} \text{Binom}(r, \kappa) \cdot ||\mu_{X_{r}(\text{R\'edei}, \kappa)} - \pi_{\text{C.L.}}||_1+ \\
\sum_{0 \leq \kappa < \epsilon \cdot r} 2 \cdot |\text{Binom}(r, \kappa)| + \sum_{(1 - \epsilon) \cdot r < \kappa \leq r} 2 \cdot |\text{Binom}(r, \kappa)|.
\end{multline*}
Thanks to Hoeffiding's inequality, the contribution from the last two summands is no more than
$$
2 \cdot \exp\left(-2 \cdot \left(\frac{1}{2} - \epsilon\right)^2 \cdot r\right),
$$
which is certainly within the bound. Hence it is enough to show that there are $C' \in \mathbb{R}_{>0}$, $\rho' \in (0,1)$ such that
$$
||\mu_{X_{r}(\text{R\'edei}, \kappa)} - \pi_{\text{C.L.}}||_1 \leq C' \rho'^r,
$$
for all $\kappa \in [\epsilon \cdot r, (1 - \epsilon) \cdot r]$.

Therefore we now focus on showing the existence of such $C'$ and $\rho'$. To this end we start by defining for each $\kappa \in \mathbb{Z}_{\geq 1}$ the following sequence of random variables
$$
Z_i(\kappa): \mathcal{S}_{\text{R\'edei}}(\kappa) \to \{0,1\}.
$$
Let $B$ be in $\mathcal{S}_{\text{R\'edei}}(\kappa)$. If $i \leq \kappa$ we put $Z_i(\kappa)(B) = 0$ in case $\omega_i(B) \cdot H_i \neq 0$ and we put $Z_i(\kappa)(B) = 1$ otherwise. Instead for $i > \kappa$ we put $Z_i(\kappa)(B) = 0$ in case the last $i - \kappa$ columns of $\omega_i(B)$ are linearly independent and the following additional condition is satisfied. We ask that for each vector $x \in \mathbb{F}_2^{[i]}$ with $\pi_{[\kappa]}(x) = H_{\kappa}$ we have that $\omega_i(B) x \neq 0$.  

We claim that $Z_i(\kappa)(B) = 0$ implies that $\text{Im}(\omega_i(B)) + \text{Im}(\omega_i(B)^T) = \mathbb{F}_2^{[i]}$. Put $q := \text{min}(i, \kappa)$. Then we always have the inclusion
\[
\text{Im}(\omega_i(B) + \omega_i(B)^T) \supseteq \{(a_1, \dots, a_i) \in \FF_2^{[i]} : a_1 + \dots + a_q = 0, a_j = 0 \text{ for } j > q\}.
\]
But since the last $i - \kappa$ columns are linearly independent, it follows that the projection map on the last $i - \kappa$ coordinates remains surjective when restricted to $\text{Im}(\omega_i(B)) + \text{Im}(\omega_i(B)^T)$. Hence $\text{Im}(\omega_i(B)) + \text{Im}(\omega_i(B)^T) \neq \mathbb{F}_2^{[i]}$ implies that
\[
\text{Im}(\omega_i(B)) + \text{Im}(\omega_i(B)^T) = \{(a_1, \dots, a_i) \in \FF_2^{[i]} : a_1 + \dots + a_q + x_{q + 1}a_{q+1} + \dots + x_ia_i = 0\}
\]
for some $x_{q + 1}, \dots, x_i \in \FF_2$. After applying $\perp$ and Remark \ref{images span everything}, we see that this is excluded by our assumptions on $H_i$ and $H_\kappa$, which establishes our claim. Hence parts (2) and (3) of Proposition \ref{transition as CL} give
$$
\mathbb{P}(X_{i + 1}(\text{R\'edei}, \kappa) = j | X_i(\text{R\'edei}, \kappa) = k, Z_i(\kappa) = 0) = Q_{\text{C.L.}}(k, j)
$$
for each $k \in \{0, \dots, i\}$ and $j \in \{0, \dots, i + 1\}$. 

Let us now bound $\mathbb{P}(Z_i(\kappa) = 1)$. In case $i \leq \kappa$ then we clearly have that this probability equals $\frac{1}{2^i}$. For $i>\kappa$ we use the union bound, applied to the $2^{i - \kappa}$ candidate vectors $x$ each happening with probability at most $\frac{1}{2^i}$, to deduce that
$$
\mathbb{P}(Z_i(\kappa) = 1) \leq \frac{1}{2^{\kappa}} + \left(1 - \prod_{j = r - \kappa + 1}^r \left(1 - \frac{1}{2^j}\right)\right).
$$ 
For $i \leq r$ and $\kappa \in [\epsilon \cdot r,(1 - \epsilon) \cdot r]$ we can bound the two summands as follows. The first summand is smaller than $\frac{1}{2^{\epsilon \cdot r}} \leq \frac{1}{2^{\epsilon \cdot i}} $. The second summand is no more than 
\[
1 - \prod_{j = r - \kappa + 1}^r \left(1 - \frac{1}{2^j}\right) \leq 1 - \left(1 - \frac{1}{2^{\epsilon r+1}}\right)^{(1 - \epsilon)r} \leq 1 - \left(\frac{1}{4}\right)^{\frac{(1 - \epsilon)r}{2^{\epsilon r+1}}}.
\]
This last expression can be bounded as $c_{\epsilon}^r \leq c_{\epsilon}^i$, for a constant $c_{\epsilon} \in (0,1)$ depending only on $\epsilon$. Keeping in mind Remark \ref{transitions at finite level} we invoke Theorem \ref{ergodic with imperfect info} and obtain precisely the desired uniform upper bound. 
\end{proof}

\subsection{R\'edei matrices in Pellian families} 
\label{Pellian families} 
We now examine R\'edei matrices that occur in the study of the solubility of the following \emph{Pellian} equations. Fix $l$ an integer such that $|l|$ is a prime congruent to $3$ modulo $4$. One then looks at the solubility of
\begin{align}
\label{ePellian}
x^2 - dy^2 = l
\end{align}
with $x,y \in \mathbb{Z}$ as $d$ varies among squarefree positive integers with $l \mid d$. 

Certainly for equation (\ref{ePellian}) to be soluble, there needs to be a solution with $x, y \in \mathbb{Q}$. Here we will study only solubility with $x, y \in \mathbb{Q}$, the transition to $\mathbb{Z}$ is made in upcoming work of the authors. As we explain below the R\'edei matrix attached to $d$ is more constrained than those appearing in Section \ref{Redei matrices}. We divide the discussion according to
$$
\text{sgn}(l), \gcd(2, \Delta_{\mathbb{Q}(\sqrt{d})/\mathbb{Q}})
$$ 
and explain for each possibility which type of matrices occur and parametrize (in a rank preserving manner) each space with a rule. Finally for the corresponding random variables we prove the analogue of Theorem \ref{limiting distribution for redei matrices} in each of these cases. 

For nonnegative integers $\kappa \leq s$, we denote by $H_s(\kappa)$ the vector of $\mathbb{F}_2^{[s]}$ whose first $\kappa$ entries are ones and the remaining entries are zeroes. Before we proceed, we shall define the R\'edei matrix attached to a squarefree integer $d$.

\begin{definition}
Let $d$ be a squarefree integer and let $D$ be the discriminant of $\Q(\sqrt{d})$. Write $q_1, \dots, q_t$ for the prime divisors of $D$. Then we can uniquely decompose
\[
\chi_d = \sum_{i = 1}^t \chi_i,
\]
where $\chi_i$ is a character with conductor a power of $q_i$. if $q_i$ is an odd prime, we have that $\chi_i$ is the quadratic character of $\Q(\sqrt{q_i^\ast})$, where $q_i^\ast$ is the unique integer satisfying $|q_i^\ast| = q_i$ and $q_i^\ast \equiv 1 \bmod 4$. If instead $q_i = 2$, we have that $\chi_i$ is the quadratic character of $\Q(\sqrt{-2})$, $\Q(\sqrt{-1})$ or $\Q(\sqrt{2})$. Then the R\'edei matrix $\textup{R\'edei}(d)$ is a $t \times t$ matrix with entries
\[
\textup{R\'edei}(d)(i, j) = \chi_j(\textup{Frob } p_i) \text{ if } i \neq j.
\]
The diagonal entries are determined by the rule that the sum of every row is zero.
\end{definition}

\begin{remark}
\label{rEGerth}
In case we fix the number of prime divisors of the discriminant, it is a fact that almost all discriminants are odd. Furthermore, in case that $d < 0$ and $d \equiv 1 \bmod 4$, we know that also the sum of every column is zero. Then removing a random row and the corresponding column from the R\'edei matrix gives a matrix satisfying the constraints as described in Subsection \ref{Redei matrices}: this follows from quadratic reciprocity. Gerth \cite{Gerth} proves equidistribution in this space of matrices and proves convergence to $\pi_{\textup{C.L.}}$ as the number of prime divisors goes to infinity. With these remarks we directly see that Theorem \ref{limiting distribution for redei matrices} is an effective version of Gerth's result. However, if we consider all squarefree integers simultaneously, one also needs to consider even discriminants.
\end{remark}

\subsection{Auxiliary matrix spaces} 
\label{Auxiliary}
In this subsection we define several matrix spaces. In the remaining paragraphs of this subsection we motivate these definitions by showing that the R\'edei matrix of $d$, such that equation (\ref{ePellian}) is soluble over $\mathbb{Q}$, naturally gives a point in one of these spaces. The matrix space depends on the value of $l$ and the parity of the discriminant of $\Q(\sqrt{d})$. Let $s$ be a positive integer and $\kappa \leq s$ be a nonnegative integer. 

We let
$$
\text{Pell}_1(s, \kappa)
$$
be the space of $(s + 1) \times (s + 1)$ matrices $A$ with coefficients in $\mathbb{F}_2$ satisfying the following constraints. The first row of $A$ must be $0$ and the sum of all the columns of $A$ equals $0$. Furthermore, for $1 \leq i < j \leq \kappa + 1$ we demand that
$$
A(i, j) = A(j ,i) + 1,
$$
while for $\kappa + 1 < j \leq s + 1$ and $1 \leq i \leq s + 1$ we demand that
$$
A(i, j) =A(j, i).
$$

We next put 
$$
\text{Pell}_2(s ,\kappa),
$$
to be the space of $(s + 1) \times (s + 1)$ matrices $A$ with coefficients in $\mathbb{F}_2$ with the following constraints. The first row of $A$ must be $H_{s + 1}(\kappa + 1)$ and the sum of all the columns of $A$ equals $0$. Finally, we ask for $1 \leq i < j \leq \kappa + 1$ that
$$
A(i, j) = A(j, i) + 1,
$$
while we ask for $\kappa + 1 < j \leq s + 1$ and $1 \leq i \leq s + 1$ that
$$
A(i, j) = A(j, i).
$$

We set
$$
\text{Pell}'_1(s, \kappa)
$$
to be the space of $(s + 2) \times (s + 2)$ matrices $A$ with coefficients in $\mathbb{F}_2$ satisfying the following constraints. The first row of $A$ is $0$ and the sum of all the columns is $0$. For $1 \leq i < j \leq \kappa + 2$, with $i, j \neq 2$ we require that
$$
A(i, j) = A(j, i) + 1
$$
and
$$
A(i, 2) = A(2, i) + \kappa + 1,
$$
while we require for $\kappa + 2 < j \leq s + 2$ and $1 \leq i \leq s + 2$ that
$$
A(i, j) = A(j, i).
$$

Let now $(a, b)$ be in $\mathbb{F}_2^2$. Finally, we put
$$
\text{Pell}_3(s, \kappa, (a, b))
$$
to be the space of $(s + 2) \times (s + 2)$ matrices $A$ with coefficients in $\mathbb{F}_2$ satisfying the following constraints. The first row of $A$ equals $H_{s + 2}(\kappa + 2)$. The projection on the first two entries of the second row equals $(a, b)$. The projection on the last $s$ entries of the second column equals $H_{s}(\kappa)$. For each $1 \leq i < j \leq \kappa + 2$ and $i, j \neq 2$ we have that
$$
A(i, j) = A(j, i) + 1.
$$
If instead $\kappa + 2 < j \leq s + 2$ and $1 \leq i \leq s + 2$, we have that
$$
A(i, j) = A(j, i).
$$

\subsubsection{Positive $l$, odd discriminant}
Enumerate the odd prime divisors of $d$ different from $l$ as $q_1, \dots, q_s$ such that precisely the first $\kappa$ of the $q_i$ are congruent $3$ modulo $4$. Represent the R\'edei matrix, $\text{R\'edei}(d)$, with the prime $l$ being the first row and the character $\chi_{-l}$ being the first column. The remaining $s$ rows and columns are numbered precisely as the $q_i$. Later, we shall also have to deal with even discriminants, in which case we always put the prime $2$ in the second row and second column. With this convention the equation
$$
x^2 -dy^2 = lz^2
$$ 
is soluble over $\mathbb{Q}$ if and only if the first row of $\text{R\'edei}(d)$ is $0$. The sum of all the columns will be zero as this is true for any R\'edei matrix. Keeping in mind quadratic reciprocity we conclude that
$$
\text{R\'edei}(d) \in \text{Pell}_1(s, \kappa).
$$

\subsubsection{Negative $l$, odd discriminant}
Maintain the notation as in the previous subsection for $d, l, s, \kappa, q_1, \dots, q_s$. Now the solubility of equation (\ref{ePellian}) over $\mathbb{Q}$ becomes equivalent to the first row being $H_{s + 1}(\kappa + 1)$. Invoking quadratic reciprocity once more, we conclude that
$$
\text{R\'edei}(d) \in \text{Pell}_2(s, \kappa).
$$

\subsubsection{Positive $l$, even discriminant}
The only possibility here is that $d$ is even, otherwise equation (\ref{ePellian}) does not have solutions in $\mathbb{Q}_2$. Examining the Legendre symbols as before we conclude that
$$
\text{R\'edei}(d) \in \text{Pell}'_1(s, \kappa).
$$

\subsubsection{Negative $l$, even discriminant}
We distinguish two cases here: $d \equiv 0 \bmod 2$ or $d \equiv 3 \bmod 4$. In the first case we have
$$
\text{R\'edei}(d) \in \text{Pell}_2(s, \kappa).
$$
Instead, in the second case, we have
$$ 
\text{R\'edei}(d) \in \text{Pell}_3(s, \kappa, (a, b)),
$$
where $a = \frac{l - 1}{4}, b = \frac{d + 1}{4}$.

\subsection{Reduction to a rule}
In this section we transform, in a rank-preserving manner, each of the spaces given in Section \ref{Auxiliary} into a rule.

\subsubsection{About $\text{Pell}_1(s, \kappa)$}
Let $\kappa \leq s$ be positive integers with $\kappa$ odd. Let $A$ be in $\text{Pell}_1(s, \kappa)$. 
Erasing the first column and row one obtains a matrix $\widetilde{A}$ whose corank equals $\text{co-rk}(A)-1$. This matrix satisfies the further constraint 
$$
\widetilde{A} H_s = \widetilde{A}^T H_s = H_s(\kappa).
$$

Observe that the second equation follows from the first since $\kappa$ is odd. 

Adding the last $s - 1$ columns to the first column and then adding the last $s - 1$ rows to the first row, we see that the space of such matrices is in a rank-preserving bijection with the space of $s \times s$ matrices whose first column and first row are both equal to $H_s(\kappa)$, while the bottom right minor equals a point of the form $\omega_{s - 1}(B)$ with $B \in \mathcal{S}_{\text{R\'edei}}(\kappa-1)$. Observe that this now makes sense also for $\kappa$ even. In other words this corresponds to the rule
$$
\mathcal{S}_{\text{Pell},1}(\kappa) := \prod_{i \in \mathbb{Z}_{\geq 0}} S_i(\text{Pell}_1, \kappa),
$$
where for $1 \leq i < \kappa$ we have that
$$
S_i(\text{Pell}_1, \kappa):=\{(v,w,c) \in \mathbb{F}_2^{[i]} \times \mathbb{F}_2^{[i]} \times \mathbb{F}_2: v+w=(0,H_{i-1}), \pi_1(w)=1\},
$$
while for $i \geq \kappa$ we have that
$$
S_i(\text{Pell}_1, \kappa):=\{(v,w,c) \in \mathbb{F}_2^{[i]} \times \mathbb{F}_2^{[i]} \times \mathbb{F}_2: v=w, \pi_1(v)=0\},
$$

and finally $S_0(\text{Pell}_1, \kappa) = \{1\}$.

To each $B \in \mathcal{S}_{\text{Pell},1}(\kappa)$ corresponds a sequence of matrices $\omega_i(B) \in \text{Mat}_{\mathbb{F}_2}([i] \times [i])$. We define a corresponding sequence of random variables
$$
X_{i}(\text{Pell}_1, \kappa): \mathcal{S}_{\text{Pell},1}(\kappa) \to \{0, \dots, i\}
$$
given by the assignment $B \mapsto \text{co-rk}(\omega_i(B))$. 

As in Subsection \ref{Redei matrices}, we put
$$
\mu_r(\text{Pell}_1) := \sum_{\kappa = 0}^r \text{Binom}(r, \kappa) \cdot \mu_{X_{r}(\text{Pell}_1, \kappa)},
$$
for each integer $r \geq 0$.

\subsubsection{About $\text{Pell}_2(s, \kappa)$}
Let $\kappa \leq s$ be positive integers with $\kappa$ even. Let $A$ be in $\text{Pell}_2(s, \kappa)$. We see that we can eliminate the second column and row of $A$ to obtain a matrix whose co-rank equals $\text{co-rk}(A) - 1$. This is a $s \times s$ matrix whose first column is $e_1$, first row is $H_{s}(\kappa)$ and whose bottom right minor is a matrix arising as $\omega_{s - 1}(B)$ with $B \in \mathcal{S}_{\text{R\'edei}}(\kappa-1)$. This corresponds to the rule

$$
\mathcal{S}_{\text{Pell}, 2}(\kappa) := \prod_{i \in \mathbb{Z}_{\geq 0}} S_i(\text{Pell}_2, \kappa),
$$
where for $1 \leq i \leq \kappa$ we have that
$$
S_i(\text{Pell}_2, \kappa):=\{(v,w,c) \in \mathbb{F}_2^{[i]} \times \mathbb{F}_2^{[i]} \times \mathbb{F}_2: v+w=H_{i}, \pi_1(w)=1\},
$$
while for $i > \kappa$ we have that
$$
S_i(\text{Pell}_2, \kappa):=\{(v,w,c) \in \mathbb{F}_2^{[i]} \times \mathbb{F}_2^{[i]} \times \mathbb{F}_2: v=w, \pi_1(v)=0\},
$$

and finally $S_0(\text{Pell}_2, \kappa) = \{1\}$. As before we let
$$
X_i(\text{Pell}_2, \kappa): \mathcal{S}_{\text{Pell}, 2}(\kappa) \to \{0, \dots, i\}
$$
be the sequence of random variables given by $B \mapsto \text{co-rk}(\omega_i(B))$. Similarly, we define for every integer $r \geq 0$
$$
\mu_r(\text{Pell}_2) := \sum_{\kappa = 0}^r \text{Binom}(r, \kappa) \cdot \mu_{X_{r}(\text{Pell}_2, \kappa)}.
$$

\subsubsection{About $\text{Pell}'_1(s, \kappa)$}
By throwing away the first row and the second column, and then adding up all the other rows to the second, we get again the rule $ \mathcal{S}_{\text{Pell},1}(\kappa)$ and random variables $X_{r}(\text{Pell}_1, \kappa)$. Hence this case does not give any new sequence of random variables.

\subsubsection{About $\text{Pell}_3(s, \kappa, (a, b))$}
Arguing as above one can reduce to the rules
$$
\mathcal{S}_{\text{Pell}, 3}(\kappa, (a, b)) := \prod_{i \in \mathbb{Z}_{\geq 0}}S_i(\text{Pell}_3, (a, b), \kappa),
$$
defined for each $a,b \in \mathbb{F}_2$ in the following manner. We put $S_0(\text{Pell}_3, (a, b), \kappa)=\{1\}$, $S_1 := \{(1, a, b)\}$. For $i \geq 2$ we put
$$
S_i(\text{Pell}_3, (a, b), \kappa):=\{(v,w,c): \pi_1(v)=0, \pi_1(w)=1, \pi_2(v)=1, \pi_j(v)+\pi_j(w)=1 \ \text{for} \ 3 \leq j\},
$$
for $i \leq \kappa$ and
$$
S_i(\text{Pell}_3, (a, b), \kappa):=\{(v,w,c): \pi_1(w)=1, \pi_1(v)=0, \pi_2(v)=0, \pi_j(v)+\pi_j(w)=0 \ \text{for} \ 3 \leq j\},
$$ 
for $i > \kappa$. As before we put

$$
X_i(\text{Pell}_3, (a, b), \kappa): \mathcal{S}_{\text{Pell}, 3}(\kappa, (a, b)) \to \{0, \dots, i\},
$$
given by the assignment $B \mapsto \text{co-rk}(\omega_i(B))$ as $B$ varies in $\mathcal{S}_{\text{Pell}, 3}(\kappa, (a, b))$.

Similarly, we define
$$
\mu_r(\text{Pell}_3, (a, b)) := \sum_{\kappa = 0}^{r-1} \text{Binom}(r-1, \kappa) \cdot \mu_{X_r(\text{Pell}_3, (a, b), \kappa)}
$$
% TODO: and this sum does also not seem to be right, since we can not have kappa = r, but at most kappa <= r - 1
for each integer $r \geq 0$ and $a, b \in \mathbb{F}_2$. 

\subsubsection{Effective convergence in the Pellian families}
We now state and prove our final result.

\begin{theorem} 
\label{limiting distribution for Pellian families}
There are $C \in \mathbb{R}_{>0}$ and $\rho \in (0,1)$ such that the following two statements hold. \\
$(1)$ We have for all integers $r \geq 0$ and for all $j \in \{1,2\}$
$$
||\mu_r(\emph{Pell}_j) - \pi_{\emph{C.L.}}||_1 \leq C \cdot \rho^r.
$$
$(2)$ We have for all integers $r \geq 0$ and for all $a,b \in \mathbb{F}_2$
$$
||\mu_r(\emph{Pell}_3, (a, b)) - \pi_{\emph{C.L.}}||_1 \leq C \cdot \rho^r.
$$
\end{theorem}

\begin{proof}
The proof is the same as the proof of Theorem \ref{limiting distribution for redei matrices} except for the choice of the variables $Z_i$. We shall only focus on this aspect, provide the upper bound for $\mathbb{P}(Z_i = 1)$ in each case, and then refer to the proof of Theorem \ref{limiting distribution for redei matrices}. Let us first prove (1).

We start with the case $j = 1$. If $i \leq \kappa$, we put $Z_i(\kappa)(B) = 0$ in case $\omega_i(B) \cdot H_i \neq 0$ and we put $Z_i(\kappa)(B) = 1$ otherwise. Instead for $i > \kappa$ we put $Z_i(\kappa)(B) = 0$ in case the bottom right $i - 1 \times i - \kappa$ submatrix of $\omega_i(B)$ has full rank and furthermore for each vector $x \in \mathbb{F}_2^{[i]}$ with $\pi_{[\kappa]}(x) = H_{\kappa}$ we have that $\omega_i(B) x \neq 0$.  

Following the proof of Theorem \ref{limiting distribution for redei matrices} one can check that $Z_i(B) = 0$ implies that $\text{Im}(\omega_i(B)) + \text{Im}(\omega_i(B)^T) = \mathbb{F}_2^{[i]}$. However we also want to guarantee that
$$
\text{Im}(\omega_i(B)) \cap \text{Im}(\omega_i(B)^T) \not \subseteq \text{ker}(\pi_1).
$$
Indeed, this ensures that we get the same transitioning probabilities as in the proof of Theorem \ref{limiting distribution for redei matrices} when we restrict to vectors whose first component is fixed. Taking $\perp$ we see that the condition is equivalent to 
$$
e_1 \not \in \text{ker}(\omega_i(B)) + \text{ker}(\omega_i(B)^T).
$$
Suppose
$$
e_1 = t_1 + t_2
$$
with $t_1 \in \text{ker}(\omega_i(B)^T)$ and $t_2 \in \text{ker}(\omega_i(B))$. If we apply $\omega_i(B)$ to the above equality, we obtain
$$
H_i(\kappa) = \omega_i(B) t_1= (\omega_i(B) + \omega_i(B)^T) t_1.
$$
But this last equality is impossible since $\text{Im}(\omega_i(B) + \omega_i(B)^T) \subseteq \text{ker}(\pi_1)$, while $H_i(\kappa)$ has first coordinate non-zero. Hence this condition is automatically satisfied, and we may proceed as in Theorem \ref{limiting distribution for redei matrices}. 

Let us now consider $j = 2$. With the same definition of $Z_i$ as in the case $j = 1$, we still have
$$
\text{Im}(\omega_i(B)) + \text{Im}(\omega_i(B)^T) = \mathbb{F}_2^{[i]}.
$$
However, we also want to guarantee that 
$$
e_1 \not \in \text{ker}(\omega_i(B)) + \text{ker}(\omega_i(B)^T).
$$ 
This time the equality becomes
\begin{align}
\label{eContra}
e_1 = (\omega_i(B) + \omega_i(B)^T) t_1.
\end{align}
The matrix $\omega_i(B) + \omega_i(B)^T$ has as top left $\text{min}(i, \kappa) \times \text{min}(i, \kappa)$ minor the matrix with zeroes on the diagonal and ones everywhere else. All other entries of the matrix $\omega_i(B) + \omega_i(B)^T$ are zero.

In case $\text{min}(i, \kappa)$ is odd, then equation (\ref{eContra}) is impossible, since the image of $\omega_i(B) + \omega_i(B)^T$ is in that case contained in the sum zero space. In case $\text{min}(i, \kappa)$ is even, we conclude that 
\[
\pi_{[\text{min}(i, \kappa)]}(t_1) = (0, H_{\text{min}(i, \kappa) - 1}).
\]
Hence it is sufficient to further demand that $\omega_i(B)^T x \neq 0$ for every vector $x$ with projection in the first $\text{min}(i, \kappa)$  coordinates equal to $(0, H_{\text{min}(i, \kappa) - 1})$. This is still at most $\frac{1}{2^{i}}$ for $i \leq \kappa$ and by the union bound no more than
$$
\frac{1}{2^{\kappa}}
$$
for $i > \kappa$. Hence with this small modification, one can proceed as in the proof of Theorem \ref{limiting distribution for redei matrices}. 

For the proof of part $(2)$, we additionally \emph{fix} the second row and then bound the conditional probabilities with the same considerations used as in part $(1)$ for $j = 2$. 
\end{proof}

\end{document}